\documentclass{article}
\pdfoutput=1 
\usepackage{amsmath, amsthm, amssymb, amstext, amsfonts}
\usepackage{enumerate}
\usepackage[applemac]{inputenc}
\usepackage{graphicx}

\theoremstyle{plain}
\newtheorem{thm}{Theorem}[section]
\newtheorem{lem}[thm]{Lemma}

\newtheorem{cor}[thm]{Corollary}
\newtheorem{prop}[thm]{Proposition}

\theoremstyle{definition}

\newtheorem{ex}{Example}

\newtheorem{prob}{Problem}

\newcount\commentno
\def\COMMENT#1{$^{<\the\commentno>}$%
     \vadjust{\vbox to 0pt{\vss\vskip-8pt\rightline{%
     \rlap{\hbox{\hskip7mm \vbox{\pretolerance=-1
     \doublehyphendemerits=0 \finalhyphendemerits=0
     \hsize40mm\tolerance=10000\eightpoint
     \lineskip=0pt\lineskiplimit=0pt
     \rightskip=0pt plus16mm\baselineskip8pt\noindent
     \hskip0pt       
     {$\langle$\the\commentno. #1$\rangle$}\endgraf}}}}\vss}}%
     \global\advance\commentno by1}%
\def\writecommentsasfootnotes{%
 \def\COMMENT{\global\advance\commentno by1\footnote{$^{<\the\commentno>}$}}%
 }
\def\nocomments{\def\COMMENT##1{}}
\def\?#1{\vadjust{\vbox to 0pt{\vss\vskip-8pt\leftline{%
     \llap{\hbox{\vbox{\pretolerance=-1
     \doublehyphendemerits=0\finalhyphendemerits=0
     \hsize16truemm\tolerance=10000\small
     \lineskip=0pt\lineskiplimit=0pt
     \rightskip=0pt plus16truemm\baselineskip8pt\noindent
     \hskip0pt        
     #1\endgraf}\hskip7truemm}}}\vss}}}
\newenvironment{txteq*}
  {
    \begin{equation*}
    \begin{minipage}[c]{0.85\textwidth} 
    \em                                
  }
  {\end{minipage}\end{equation*}\ignorespacesafterend}
\def\specrel#1#2{\mathrel{\mathop{\kern0pt #1}\limits_{#2}}}
\def\Specrel#1#2{\mathrel{\mathop{\kern0pt #1}\limits^{#2}}}

\newcommand{\lk}{$({<k)}$}
\newcommand{\kinsep}{\lk-insep\-ar\-able}
\newcommand{\lek}{$({< k+1)}$}
\newcommand{\lekinsep}{\lek-inseparable}
\newcommand{\lell}{$({<\ell)}$}
\newcommand{\linsep}{\lell-insep\-ar\-able}

\newcommand{\N}{\ensuremath{\mathbb{N}}}

\newcommand{\sm}{\ensuremath{\smallsetminus}}

\newcommand{\Aut}{\textnormal{Aut}}

\newcommand{\es}{\ensuremath{\emptyset}}

\newcommand{\sub}{\subseteq}

\newcommand{\cB}{\ensuremath{\mathcal B}}

\newcommand{\cD}{\ensuremath{\mathcal D}}

\newcommand{\cL}{\ensuremath{\mathcal L}}

\newcommand{\cS}{{\ensuremath S}}
\newcommand{\cT}{{\ensuremath T}}

\newcommand{\cV}{\ensuremath{\mathcal V}}

\newcommand{\TV}{\ensuremath{(\cT,\cV)}}

\newcommand{\Tshaped}{{\rm\sffamily T}}

\def\td{tree-decom\-po\-si\-tion}

\newcommand{\dhk}{\frac{3}{2}k}

\newcommand{\sepn}[2]{\ensuremath{{(#1,#2)}}}
\newcommand{\AB}{\sepn AB}
\newcommand{\BA}{\sepn BA}
\newcommand{\CD}{\sepn CD}

\let\doublebar=\|

\newcommand{\mcm}[3]{\newcommand{#1}[#2]{{\ensuremath{#3}}}} 

\mcm{\tuple}{1}{\langle #1 \rangle}
\mcm{\name}{1}{\ulcorner #1 \urcorner}
\mcm{\Nbb}{0}{\mathbb{N}}
\mcm{\Zbb}{0}{\mathbb{Z}}
\mcm{\Rbb}{0}{\mathbb{R}}
\mcm{\Cbb}{0}{\mathbb{C}}
\mcm{\Fbb}{0}{\mathbb{F}}
\mcm{\Bcal}{0}{\cal B}
\mcm{\Ccal}{0}{\cal C}
\mcm{\Dcal}{0}{\cal D}
\mcm{\Ecal}{0}{\cal E}
\mcm{\Fcal}{0}{\cal F}
\mcm{\Gcal}{0}{\cal G}
\mcm{\Hcal}{0}{\cal H}
\mcm{\Ical}{0}{\cal I}
\mcm{\Lcal}{0}{\cal L}
\mcm{\Mcal}{0}{\cal M}
\mcm{\Ncal}{0}{\cal N\!}
\mcm{\Ocal}{0}{\cal O}
\mcm{\Pcal}{0}{{\cal P}}
\mcm{\Scal}{0}{{\cal S}}
\mcm{\Tcal}{0}{{\cal T}}
\mcm{\Ucal}{0}{{\cal U}}
\mcm{\Vcal}{0}{{\cal V}}
\mcm{\Wcal}{0}{{\cal W}}
\mcm{\Ycal}{0}{{\cal Y}}
\mcm{\Mfrak}{0}{\mathfrak M}

\usepackage{verbatim}

\nocomments

\title{$k$-Blocks: a connectivity invariant for graphs}
\author{J.\ Carmesin \and R.\ Diestel \and M.\ Hamann \and F.\ Hundertmark}

\begin{document}

\maketitle

\begin{abstract}\noindent
A {\it $k$-block\/} in a graph $G$ is a maximal set of at least $k$ vertices no two of which can be separated in $G$ by fewer than $k$ other vertices. The {\em block number\/} $\beta(G)$ of $G$ is the largest integer~$k$ such that $G$ has a $k$-block.

We investigate how $\beta$ interacts with density invariants of graphs, such as their minimum or average degree. We further present algorithms that decide whether a graph has a $k$-block, or which find all its $k$-blocks.

The connectivity invariant $\beta(G)$ has a dual width invariant, the {\em block-width\/} ${\rm bw}(G)$ of~$G$. Our algorithms imply the duality theorem $\beta = {\rm bw}$: a~graph has a {\em block-decomposition\/} of width and adhesion $< k$ if and only if it contains no $k$-block.
\end{abstract}

\section{Introduction}

Given $k\in\N$, a set $I$ of at least $k$ vertices of a graph $G$ is \emph{\kinsep} if no set $S$ of fewer than $k$ vertices of~$G$ separates any two vertices of~$I\sm S$ in~$G$.
A maximal \kinsep\ set is a \emph{$k$-block}. The degree of connectedness of such a set of vertices is thus measured in the ambient graph~$G$, not only in the subgraph they induce. While the vertex set of a $k$-connected subgraph of~$G$ is clearly \kinsep\ in~$G$, there can also be $k$-blocks that induce few or no edges.

The $k$-blocks of a graph were first studied by Mader~\cite{mader78}. They have recently received some attention because, unlike its $k$-connected subgraphs, they offer a meaningful notion of the `$k$-connected pieces' into which the graph may be decomposed~\cite{confing}. This notion is related to, but not the same as, the notion of a tangle in the sense of Robertson and Seymour~\cite{GMX}; see Section~\ref{sec_Tangles} and~\cite{profiles} for more on this relationship.

Although Mader~\cite{mader72} had already proved that graphs of average degree at least~$4(k-1)$ have $k$-connected subgraphs, and hence contain a $k$-block, he did not in~\cite{mader78}
 consider the analogous extremal problem for the weaker notion of a $k$-block directly. 

Our aim in this paper is to study this problem: we ask what average or minimum degree conditions force a given finite graph to contain a $k$-block.

This question can, and perhaps should, be seen in a wider extremal context. Let $\beta(G)$ denote the {\em block number\/} of~$G$, the greatest integer~$k$ such that $G$ has a $k$-block (equivalently: has a \kinsep\ set of vertices). This $\beta$ seems to be an interesting graph invariant%
   \footnote{For example, in a network $G$ one might think of the nodes of a $\beta(G)$-block as locations to place some particularly important servers that should still be able to communicate with each other when much of the network has failed.},
  and one may ask how it interacts with other graph invariants, not just the average or minimum degree. Indeed, the examples we describe in Section~\ref{sec_Examples} will show that containing a $k$-block for large $k$ is compatible with having bounded minimum and average degree, even in all subgraphs.%
   \COMMENT{}
   So $k$-blocks can occur in very sparse graphs, and one will need bounds on other graph invariants than $\delta$ and~$d$ to force $k$-blocks in such graphs.

There is an invariant dual to~$\beta$: the least integer $k$ such that a graph~$G$ has a {\em block-decomposition\/} of adhesion and width both at most~$k$. Calling this $k$ the {\em block-width\/} ${\rm bw}(G)$ of~$G$, we can express the duality neatly as $\beta = \rm bw$.

All the graphs we consider are finite.
Our paper is organized as follows. In Section~\ref{basics} we introduce whatever terminology is not covered in~\cite{DiestelBook10noEE}, and give some background on tree-decompositions. In Section~\ref{sec_Examples} we present examples of $k$-blocks, aiming to exhibit the diversity of the concept. In Section~\ref{sec_MinDeg} we prove that graphs of minimum degree at least $2(k-1)$ have a $k$-block. If the graph~$G$ considered is ($k-1$)-connected, the minimum degree needed comes down to at most~$\frac{3}{2}(k-1)$, and further to $k$ if $G$ contains no triangle. In Section~\ref{sec_AvDeg} we show that graphs of average degree at least $3(k-1)$ contain a $k$-block. In Section~\ref{sec_Tangles} we clarify the relationship between $k$-blocks and tangles. In Section~\ref{sec_Algo} we present a polynomial-time algorithm that decides whether a given graph has a $k$-block, and another that finds all the $k$-blocks in a graph. This latter algorithm gives rise to our duality theorem $\beta = \rm bw$.

\section{Terminology and background}\label{basics}
All graph-theoretic terms not defined within this paper are explained in~\cite{DiestelBook10noEE}. Given a graph $G=(V,E)$, an ordered pair $(A,B)$ of vertex sets such that $A\cup B = V$ is called a \emph{separation} of $G$ if there is no edge $xy$ with $x \in A\sm B$ and $y\in B\sm A$. The sets $A,B$ are the \emph{sides} of this separation.
A separation $(A,B)$ such that neither $A \sub B$ nor $B\sub A$ is a \emph{proper} separation. The \emph{order} of a separation \AB\ is the cardinality of its \emph{separator} $A\cap B$. A separation of order $k$ is called a \emph{$k$-separation}. A simple calculation yields the following:

\begin{lem}\label{counting}
Given any two separations $(A,B)$ and $(C,D)$ of~$G$, the orders of the separations $(A\cap C, B\cup D)$ and $(B\cap D, A\cup C)$ sum to $|A\cap B|+|C\cap D|$.\qed
\end{lem}

Recall that a \emph{\td} of $G$ is a pair $(\cT,\cV)$ of a tree $\cT$ and a family $\cV=(V_t)_{t\in \cT}$ of vertex sets $V_t\sub V$, one for every node of~$\cT$, such that:
\begin{enumerate}[(T1)]
\item $V = \bigcup_{t\in \cT}V_t$;
\item for every edge $e\in G$ there exists a $t \in \cT$ such that both ends of $e$ lie in~$V_t$;
\item $V_{t_1} \cap V_{t_3} \sub V_{t_2}$ whenever $t_2$ lies on the $t_1$--$t_3$ path in~$\cT$.
\end{enumerate}

\noindent
The sets $V_t$ are the \emph{parts} of \TV, its \emph{width} is the number $\max_{t\in\cT} |V_t|-1$, and the \emph{tree-width} of~$G$ is the least width of any \td\ of~$G$.

The intersections $V_t\cap V_{t'}$ of ‘adjacent’ parts in a tree-decomposition $(\cT,\cV)$ (those for which $tt'$ is an edge of~$\cT$) are its \emph{adhesion sets}; the maximum size of such a set is the \emph{adhesion} of $(\cT,\cV)$. The \emph{interior} of a part $V_t$, denoted by $\mathring{V}_t$, is the set of those vertices in~$V_t$ that lie in no adhesion set. By~(T3), we have $\mathring{V}_t = V_t\sm\bigcup_{t'\neq t}V_{t'}$.

Given an edge $e=t_1t_2$ of~$\cT$, the two components $\cT_1\owns t_1$ and $\cT_2\owns t_2$ of $\cT-e$ define separations \AB\ and \BA\ of~$G$ with $A = \bigcup_{t\in\cT_1} V_t$ and $B=\bigcup_{t\in\cT_2} V_t$, whose separator is the adhesion set~$V_{t_1}\cap V_{t_2}$ \cite[Lemma~12.3.1]{DiestelBook10noEE}. We call these the separations \emph{induced} by the \td\ \TV. Note that the adhesion of a \td\ is the maximum of the orders of the separations it induces.

A~\td\ {\em distinguishes\/} two $k$-blocks $b_1,b_2$ if it induces a separation that separates them. It does so {\em efficiently\/} if this separation can be chosen of order no larger than the minimum order of a $b_1$--$b_2$ separator in~$G$. The \td\ $(\cT, \cV)$ is $\Aut(G)$-{\em invariant\/} if the automorphisms of $G$ act on the set of parts in a way that induces an action on the tree~\cT. The following theorem was proved in~\cite{confing}:

\begin{thm}\label{canonicaltd}
For every $k\in\N$, every graph $G$ has an $\Aut(G)$-invariant \td\ of adhesion at most~$k$ that efficiently distinguishes all its $k$-blocks.
\end{thm}

A \td\ $(\cT,\Vcal)$ of a graph $G$ is \emph{lean} if for any nodes $t_1,t_2\in \cT$, not necessarily distinct,
and vertex sets $Z_1\subseteq V_{t_1}$ and $Z_2\subseteq V_{t_2}$ such that ${|Z_1|=|Z_2|=:\ell}$, either $G$ contains
$\ell$ disjoint $Z_1$--$Z_2$ paths or there exists an edge $tt'\in t_1\cT t_2$ with $|V_t\cap V_{t'}|<\ell$. Since there is no such edge when $t_1=t_2 =:t$, this implies in particular that, for every part~$V_t$, any two subsets $Z_1,Z_2\sub V_t$ of some equal size~$\ell$ are linked in $G$ by $\ell$ disjoint paths.%
   \COMMENT{}
(However, the parts need not be \linsep\ for any large~$\ell$; see Section~\ref{sec_Examples}.)

We call a \td\ $(\cT,\Vcal)$ \emph{$k$-lean} if none of its parts contains another, it has adhesion at most~$k$, and for any nodes $t_1,t_2\in \cT$,  not necessarily distinct,
and vertex sets $Z_1\subseteq V_{t_1}$ and $Z_2\subseteq V_{t_2}$ such that $|Z_1|=|Z_2|=:\ell\leq k+1$, either $G$ contains
$\ell$ disjoint $Z_1$--$Z_2$ paths or there exists an edge $tt'\in t_1\cT t_2$ with $|V_t\cap V_{t'}|<\ell$.

Thomas~\cite{thomas90} proved that every graph~$G$ has a lean \td\ whose width is no greater than the tree-width of~$G$. By considering only separations of order at most~$k$ one can adapt the short proof of Thomas's theorem given in~\cite{bellenbaumDiestel} to yield the following:

\begin{thm}\label{k-lean}
 For every $k\in\N$, every graph has a $k$-lean \td.
\end{thm}


\section{\boldmath Examples of $k$-blocks}\label{sec_Examples}

In this section we discuss three different types of $k$-block.

\begin{ex}\label{kconex}
The vertex set of any $k$-connected subgraph is \kinsep, and hence contained in a $k$-block.
\end{ex}

While a $k$-block as in Example~\ref{kconex} derives much or all of its inseparability from its own connectivity as a subgraph, the $k$-block in our next example will form an independent set. It will derive its inseparability from the ambient graph, a large grid to which it is attached.

\begin{ex}\label{ex_gridBlock}
Let $k\geq 5$, and let $H$ be a large $(m\times n)$-grid, with $m,n\geq {k^2}$ say.%
   \COMMENT{}   
Let $G$ be obtained from~$H$ by adding a set $X = \{x_1,\ldots,x_k\}$ of new vertices, joining each $x_i$ to at least $k$ vertices on the grid boundary that form a (horizontal or vertical) path in~$H$%
   \COMMENT{}
    so that every grid vertex obtains degree~$4$ in~$G$ (Figure~\ref{pic_gridEx}).
We claim that $X$ is a $k$-block of~$G$, and is its only $k$-block.

Any grid vertex can lie in a common $k$-block of~$G$ only with its neighbours, because these separate it from all the other vertices. As any $k$-block has at least $k\ge 5$ vertices but among the four $G$-neighbours of a grid vertex at least two are non-adjacent grid vertices, this implies that no $k$-block of~$G$ contains a grid vertex. On the other hand, every two vertices of $X$ are linked by $k$ independent paths in~$G$, and hence cannot be separated by fewer than~$k$ vertices. Hence $X$ is \kinsep, maximally so, and is thus the only $k$-block of~$G$.
\end{ex}

\begin{figure}[h]
\begin{center}
\includegraphics[width=.6\textwidth]{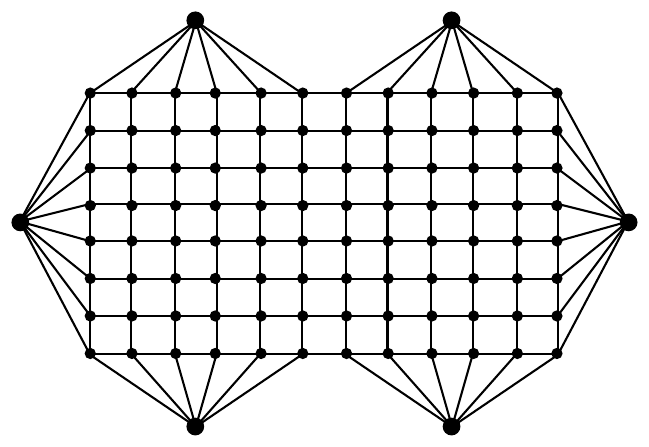}
\caption[Figure 1]{The six outer vertices form a $6$-block}\label{pic_gridEx}
   \vskip-12pt\vskip-12pt
\end{center}
\end{figure}

In the discussion of Example~\ref{ex_gridBlock} we saw that none of the grid vertices lies in a $k$-block. In particular, the grid itself has no $k$-block when $k\ge 5$. Since every two inner vertices of the grid, those of degree~4, are joined in the grid by 4 independent paths, they form a ${(<4)}$-inseparable set (which is clearly maximal):

\begin{ex}\label{grid}
The inner vertices of any large grid $H$ form a 4-block in~$H$. However, $H$ has no $k$-block for any $k\ge 5$.
\end{ex}

The $k$-block defined in Example~\ref{ex_gridBlock} gives rise to a tangle of large order (see Section~\ref{sec_Tangles}), the same as the tangle specified by the grid~$H$. This is in contrast to our last two examples, where the inseparability of the $k$-block will again lie in the ambient graph but in a way that need not give rise to a non-trivial tangle. (See Section~\ref{sec_Tangles} for when it does.) Instead, the paths supplying the required connectivity will live in many different components of the subgraph into which the $k$-block splits the original graph.

\begin{ex}\label{TKn}
Let $X$ be a set of $n\ge k$%
   \COMMENT{}
   isolated vertices. Join every two vertices of~$X$ by many (e.g., $k$) independent paths, making all these internally disjoint. Then $X$ will be a $k$-block in the resulting graph.
\end{ex}

Example~\ref{TKn} differs from Example~\ref{ex_gridBlock} in that its graph has a \td\ whose only part of order~$\ge 3$ is $X$. Unlike the grid in Example~\ref{ex_gridBlock}, the paths providing $X$ with its external connectivity do not between them form a subgraph that is in any sense highly connected. We can generalize this as follows:

\begin{ex}\label{ex_treeBlock}
Given $n\ge k$, consider a tree $T$ in which every non-leaf node has $\binom{n}{k-1}$ successors. Replace each node $t$ by a set $V_t$ of $n$ isolated vertices. Whenever $t'$ is a successor of a node~$t$ in~$T$, join $V_{t'}$ to a ($k-1$)-subset $S_{t'}$ of~$V_t$ by $(k-1)$ independent edges, so that these $S_{t'}$ are distinct sets for different successors $t'$ of~$t$. For every leaf $t$ of~$T$, add edges on $V_t$ to make it complete. The $k$-blocks of the resulting graph~$G$ are all the sets~$V_t$ ($t\in T$), but only the sets $V_t$ with $t$ a leaf of $T$ induce any edges.
\end{ex}

Interestingly, the $k$-blocks that we shall construct explicitly in our proofs will all be {\em connected\/}, i.e., induce connected subgraphs. Thus, our proof techniques seem to be insufficient to detect $k$-blocks that are disconnected or even independent, such as those in our examples. However, we do not know whether or not this affects the quality of our bounds or just their witnesses:

\begin{prob}\label{condiscon}
Does every minimum or average degree bound that forces the existence of a $k$-block also force the existence of a connected \kinsep\ set?%
   \COMMENT{}
   \end{prob}

Even if the answer to this problem is positive, it will reflect only on how our invariant~$\beta$ relates to the invariants $\delta$ and~$d$, and that for some graphs it may be more interesting to relate $\beta$ to other invariants. The existence of a large $k$-block in Examples \ref{ex_gridBlock} and~\ref{TKn}, for instance, will not follow from any theorem relating $\beta$ to $\delta$ or~$d$, since graphs of this type have a bounded average degree%
   \COMMENT{}
   independent of~$k$, even in all subgraphs. But they are key examples, which similar results about $\beta$ and other graph invariants may be able to detect.


\section{\boldmath Minimum degree conditions forcing a $k$-block}\label{sec_MinDeg}

Throughout this section, let $G=(V,E)$ be a fixed non-empty graph. We ask what minimum degree will force $G$ to contain a $k$-block for a given integer~$k>0$.

Without any further assumptions on~$G$ we shall see that $\delta(G)\ge 2(k-1)$ will be enough.
If we assume that $G$ is ($k-1$)-connected~-- an interesting case, since for such $G$ the parameter~$k$ is minimal such that looking for $k$-blocks can be non-trivial~-- we find that $\delta(G) > \frac{3}{2}k-\frac{5}{2}$ suffices. If $G$ is ($k-1$)-connected but contains no triangle, even $\delta(G)\ge k$ will be enough. Note that this is best possible in the (weak) sense that the vertices in any $k$-block will have to have degree at least~$k$, except in some very special cases that are easy to describe.%
   \COMMENT{}

Conversely, we construct a ($k-1$)-connected graph of minimum degree ${\lfloor \dhk-\frac{5}{2}\rfloor}$ that has no $k$-block. So our second result above is sharp.

\medbreak

To enhance the readability of both the results and the proofs in this section, we give bounds on $\delta$ which force the existence of a ($k+1$)-block for any $k\ge 0$. 

We shall often use the fact that a vertex of $G$ together with $k$ or more of its neighbours forms a \lekinsep\ set as soon as these neighbours are pairwise not separated by $k$ or fewer vertices. Let us state this as a lemma:

\begin{lem}\label{vx_block}
Let $v\in V$ and $N\sub N(v)$ with $|N|\geq k$.
If no two vertices of~$N$ are separated in $G$ by at most~$k$ vertices, then $N\cup\{v\}$ lies in a $(k+1)$-block.\qed
\end{lem}

Here, then, is our first sufficient condition for the existence of a $k$-block. It is essentially due to Mader~\cite[{Satz~$7'$}]{mader74}, though with a different proof:

\begin{thm}\label{min_deg}
If $\delta(G)\ge 2k$, then $G$ has a $(k+1)$-block. This $(k+1)$-block can be chosen to be connected in~$G$ and of size at least~$\delta(G)+1-k$.%
   \COMMENT{}
\end{thm}

\begin{proof}{\lineskiplimit=-3pt
If $k=0$, then the assertion follows directly.
So we assume $k>0$.
By Theorem~\ref{k-lean}, $G$ has a $k$-lean \td\ $(\cT,\Vcal)$, say with $\cV = (V_t)_{t\in \cT}$. Pick a leaf $t$ of~$\cT$. (If $\cT$ has only one node, we count it as a leaf.) Write $A_t := V_t\cap\bigcup_{t'\ne t} V_{t'}$ for the attachment set of~$V_t$. As $V_t$ is not contained in any other part of $(\cT,\Vcal)$,%
   \COMMENT{}
   we have $\mathring{V}_t = V_t\sm A_t \ne\es$ by~(T3). By our degree assumption and $|A_t|\le k$, every vertex in $\mathring{V}_t$ has $k$ neighbours in~$\mathring{V}_t$. Thus, $|\mathring{V}_t|\ge k+1\ge 2$.%
   \COMMENT{}

}We prove that $\mathring{V}_t$ extends to a $(k+1)$-block $B\sub V_t$ that is connected in~$G$. Pick distinct vertices $v,v'\in\mathring{V}_t$. Let $N$ be a set of $k$ neighbours of~$v$, and $N'$ a set of $k$ neighbours of~$v'$. Note that $N\cup N'\sub V_t$.
As our \td\ is $k$-lean, there are $k+1$ disjoint paths in $G$ between the $(k+1)$-sets $N\cup\{v\}$ and $N'\cup\{v'\}$. Hence $v$ and $v'$ cannot be separated in $G$ by at most~$k$ other vertices.%
   \COMMENT{}

We have thus shown that $\mathring{V}_t$ is \lekinsep.%
   \COMMENT{}
  In particular, $A_t$~does not separate it, so $\mathring{V}_t$ is connected in~$G$. Let $B$ be a $(k+1)$-block containing~$\mathring{V}_t$. As $A_t$ separates $\mathring{V}_t$ from $G\sm V_t$, we have $B\sub V_t$. Every vertex of $B$ in~$A_t$ sends an edge to~$\mathring{V}_t$, since otherwise the other vertices of~$A_t$ would separate it from~$\mathring{V}_t$. Hence $B$ is connected. Since every vertex in $\mathring{V}_t$ has at least $\delta(G)-k$ neighbours in~$\mathring{V}_t\sub B$, we have the desired bound of $|B|\ge\delta(G)+1-k$.
\end{proof}

One might expect that our lower bound for the size of the $(k+1)$-block $B$ found in the proof of Theorem~\ref{min_deg} can be increased by proving that $B$ must contain the adhesion set of the part~$V_t$ containing it. While we can indeed raise the bound a little (by at least~1, but we do not know how much at most), we show in%
   \ Section~\ref{sec_FurtherEx}
that $B$ can lie entirely in the interior of~$V_t$.

We also do not know whether the degree bound of $\delta(G)\ge 2k$ in Theorem~\ref{min_deg} is sharp. The largest minimum degree known of a graph without a $(k+1)$-block is $\lfloor\dhk-1\rfloor$. This graph (Example~\ref{lange_Wurst} below) is $k$-connected, and we shall see that $k$-connected graphs of larger minimum degree do have $(k+1)$-blocks (Theorem~\ref{thm_m_wat}). Whether or not graphs of minimum degree between $\dhk-1$ and $2k$ and connectivity $<k$ must have $(k+1)$-blocks is unknown to us:

\begin{prob}\label{oque}
Given~$k\in\N$, determine the smallest value $\delta_k$ of $\delta$ such that every graph of minimum degree at least~$\delta$ has a $k$-block.
\end{prob}

It is also conceivable that the smallest minimum degree that will force a {\em connected\/} $(k+1)$-block~-- or at least a connected \lekinsep\ set, as found by our proof of Theorem~\ref{min_deg}~-- is indeed~$2k$ but possibly disconnected $(k+1)$-blocks can be forced by a smaller value of~$\delta$ (compare Problem~\ref{condiscon}).

\medbreak

The degree bound of Theorem~\ref{min_deg} can be reduced by imposing additional conditions on~$G$. Our next aim is to derive a better bound on the assumption that $G$ is $k$-connected, for which we need a few lemmas.

We say that a $k$-separation $(A,B)$ is \emph{\Tshaped-shaped} (Fig.~\ref{pic_Tshaped}) if it is a proper separation and there exists another proper $k$-separation $(C,D)$ such that ${A\sm B}\sub C\cap D$ as well as $|A\cap C|\leq k$ and ${|A\cap D|\leq k}$.
Obviously, $(A,B)$ is \Tshaped-shaped witnessed by $(C,D)$ if and only if the two separations $(A\cap C,B\cup D)$ and $(A\cap D,B\cup C)$ have order at most $k$ and are improper separations.%
   \COMMENT{}

\begin{figure}[h]
\begin{center}
\includegraphics[width=.6\textwidth]{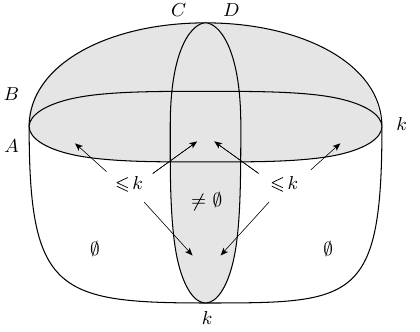}
\caption[Figure 1]{The separation \AB\ is \Tshaped-shaped}\label{pic_Tshaped}
   \vskip-12pt\vskip-0pt
\end{center}
\end{figure}

\begin{lem}\label{link_small}
If $(A,B)$ is a \Tshaped-shaped $k$-separation in~$G$, then $|A|\leq\dhk$.
\end{lem}

\begin{proof}
Let $(C,D)$ witness that $(A,B)$ is \Tshaped-shaped. Then
\[\textstyle
|A|\leq |A\cap B|+|(C\cap D)\sm B|\leq k+{1\over2}(2k-k)=\dhk.\qedhere%
   \COMMENT{}
\]
\end{proof}

When a $k$-separation \AB\ is \Tshaped-shaped, no $(k+1)$-block of~$G$ can lie in~$A$: with \CD\ as above, it would have to lie in either $A\cap C$ or~$A\cap D$, but both these are too small to contain a $(k+1)$-block. Conversely, one may ask whether every proper $k$-separation \AB\ in a $k$-connected graph such that $A$ contains no $(k+1)$-block must be \Tshaped-shaped, or at least give rise to a \Tshaped-shaped $k$-separation $(A',B')$ with $A'\sub A$.\vadjust{\penalty-5000} This, however, is not true: some counterexamples are given in 
     Section~\ref{sec_FurtherEx}.

Interestingly, though, a global version of this does hold: a \Tshaped-shaped $k$-sep\-ar\-ation must occur \emph{somewhere} in every $k$-connected graph that has no $(k+1)$-block. More precisely, we have the following:%
   \COMMENT{}

\begin{lem}\label{lem_T-shaped}
If $G$ is $k$-connected, the following statements are equivalent:\vskip-6pt\vskip-6pt
\begin{enumerate}[\rm (i)]\itemsep0em
\item\label{itm1_T-shaped} every proper $k$-separation of $G$ separates two $(k+1)$-blocks;
\item\label{itm2_T-shaped} no $k$-separation of $G$ is \Tshaped-shaped.
\end{enumerate}
\end{lem}

\begin{proof}
We first assume~(\ref{itm1_T-shaped}) and show~(\ref{itm2_T-shaped}). If (\ref{itm2_T-shaped}) fails, then $G$~has a $k$-separation $(A,B)$ that is \Tshaped-shaped, witnessed by $(C,D)$ say. We shall derive a contradiction to~(i) by showing that $A$ contains no $(k+1)$-block. If $A$ contains a $(k+1)$-block, it lies in either $A\cap C$ or~$A\cap D$, since no two of its vertices are separated by $(C,D)$. 
By the definition of \Tshaped-shaped, none of these two cases can occur, a contradiction.

Let us now assume~(\ref{itm2_T-shaped}) and show~(\ref{itm1_T-shaped}). If (\ref{itm1_T-shaped}) fails, there is a proper $k$-separation $(A,B)$ such that $A$ contains no $(k+1)$-block.
Pick such an $(A,B)$ with $|A|$ minimum. Since $(A,B)$ is proper, there is a vertex $v\in A\sm B$.
Since $G$ is $k$-connected, $v$ has at least $k$ neighbours, all of which lie in~$A$.
As $A$ contains no $(k+1)$-block, Lemma~\ref{vx_block} implies that there is a proper $k$-separation $(C,D)$ that separates two of these neighbours. Then $v$ must lie in $C\cap D$.

We first show that either $(A\cap C, B\cup D)$ 
has order at most $k$ and $(A\cap C)\sm(B\cup D)=\es$ or $(B\cap D, A\cup C)$ has order at most $k$ and $(B\cap D)\sm(A\cup C)=\es$.
Let us assume that the first of these fails; then either $(A\cap C,B\cup D)$ has order~$>k$ or $(A\cap C)\sm(B\cup D)\ne\es$. In fact, if the latter holds then so does the former: otherwise $(A\cap C, B\cup D)$ is a proper $k$-separation that contradicts the minimality of $|A|$ in the choice of~$(A,B)$. (We have $|A\cap C| < |A|$, since $v$ has a neighbour in $A\sm C$.)%
   \COMMENT{}
  Thus, $(A\cap C,B\cup D)$ has order~$>k$. As $|A\cap B|+|C\cap D|=2k$, this implies by Lemma~\ref{counting} that the order of $(B\cap D, A\cup C)$ is strictly less than~$k$. As $G$ is $k$-connected, this means that $(B\cap D, A\cup C)$ is not a proper separation, i.e., that $(B\cap D)\sm(A\cup C)=\es$ as claimed.

By symmetry, we also get the analogous statement for the two separations $(A\cap D,B\cup C)$ and $(B\cap C,A\cup D)$. But this means that one of the separations $(A,B)$, $(B,A)$, $(C,D)$ and $(D,C)$ is \Tshaped-shaped,%
   \COMMENT{}
   contradicting~(\ref{itm2_T-shaped}).%
   \COMMENT{}
\end{proof}

Our next lemma says something about the size of the $(k+1)$-blocks we shall find.

\begin{lem}\label{lem_sepSepLargeKBlock}
If $G$ is $k$-connected and $|A|>\dhk$ for every proper $k$-separa\-tion $(A,B)$ of~$G$, then either $V$ is a $(k+1)$-block or $G$ has two $(k+1)$-blocks of size at least $\min\{\,|A| : (A,B)\text{ is a proper }k\text{-separation}\,\}$ that are connected in~$G$.
\end{lem}%
   \COMMENT{}

\begin{proof}
By assumption and Lemma \ref{link_small}, $G$~has no \Tshaped-shaped $k$-separation, so by Lemma~\ref{lem_T-shaped} every side of a proper $k$-separation contains a $(k+1)$-block.

By Theorem~\ref{k-lean}, $G$~has a $k$-lean \td\ $(\cT,\cV)$, with $\cV = (V_t)_{t\in \cT}$ say. Unless $V$ is a $(k+1)$-block, in which case we are done, this decomposition has at least two parts: since there exist two $(k+1)$-sets in $V$ that are separated by some $k$-separation,%
   \COMMENT{}
   the trivial \td\ with just one part would not be $k$-lean.

So $\cT$ has at least two leaves, and for every leaf $t$ the separation $\AB := \big(V_t, \bigcup_{t' \neq t}V_{t'}\big)$ is a proper $k$-separation. It thus suffices to show that $A=V_t$ is a $(k+1)$-block;%
   \COMMENT{}
  it will clearly be connected (as in the proof of Theorem~\ref{min_deg}).%
   \COMMENT{}

As remarked at the start of the proof, there exists a $(k+1)$-block $X\sub A$. If $X\ne A$, then $A$ has two vertices that are separated by a $k$-separation $(C,D)$; we may assume that $X\sub C$, so $X\sub A\cap C$.

If $(A\cap C,B\cup D)$ has order $\le k$, it is a proper separation (as $X\sub A\cap C$ has size~$>k$);%
   \COMMENT{}
   then its separator $S$ has size exactly~$k$, since $G$ is $k$-connected.
By the choice of~$(C,D)$ there is a vertex $v$ in $(D\sm C)\cap A$. The $k+1$ vertices of $S\cup\{v\}\sub A$ are thus separated in $G$ by the $k$-set $C\cap D$ from $k+1$ vertices in $X\sub A\cap C$, which contradicts the leanness of~$(\cT,\cV)$ for $V_t=A$.

So the order of $(A\cap C,B\cup D)$ is at least $k+1$. By
Lemma~\ref{counting}, the order of $(B\cap D,A\cup C)$ must then be less than $k$, so by the $k$-connectedness of $G$ there is no $(k+1)$-block in $B\cap D$.
   \COMMENT{}

The $(k+1)$-block $X'$ which $D$ contains (see earlier) thus lies in~$D\cap A$. So $A$ contains two $(k+1)$-blocks $X$ and~$X'$, and hence two vertex sets of size $k+1$, that are separated by $\CD$, which contradicts the $k$-leanness of~$\TV$.%
   \COMMENT{}%
   \COMMENT{}
\end{proof}

\begin{thm}\label{thm_m_wat}
If $G$ is $k$-connected and $\delta(G)> \dhk-1$,%
   \COMMENT{}
   then either $V$ is a $(k+1)$-block\penalty-200\ or $G$ has at least two $(k+1)$-blocks. These can be chosen to be connected in~$G$ and of size at least $\delta(G)+1$.
\end{thm}

\begin{proof}
For every proper $k$-separation $(A,B)$ we have a vertex of degree~$>\dhk-1$ in $A\sm B$, and hence $|A|\geq\delta(G)+1>\dhk$. The assertion now follows from Lemma~\ref{lem_sepSepLargeKBlock}.
\end{proof}

To show that the degree bound in Theorem~\ref{thm_m_wat} is sharp, let us construct a $k$-connected graph $H$ with $\delta(H) = \lfloor \dhk-1\rfloor$ that has no $(k+1)$-block.

\begin{ex} \label{lange_Wurst}
Let $H_n$ be the ladder that is a union of $n\geq 2$ squares (formally: the cartesian product of a path of length~$n$ with a~$K^2$).

For even $k$, let $H$ be the lexicographic product of~$H_n$ and a complete graph $K=K^{k/2}$, i.e., the graph with vertex set $V(H_n)\times V(K)$ and edge set
\[\{\,(h_1,x)(h_2,y)\mid \text{either }h_1=h_2\text{ and }xy\in E(K)\text{ or }h_1h_2\in E(H_n)\,\},\]
see Figure~\ref{pic_Wurscht}. This graph $H$ is $k$-connected and has minimum degree~$\dhk-1$. But it contains no $(k+1)$-block: among any $k+1$ vertices we can find two that are separated in~$H$ by a $k$-set of the form $V_{h_1}\cup V_{h_2}$, where $V_h := \{(h,x)\mid x\in K\}$.%
   \COMMENT{}

\begin{figure}[h]
\begin{center}
\includegraphics[width=.75\textwidth]{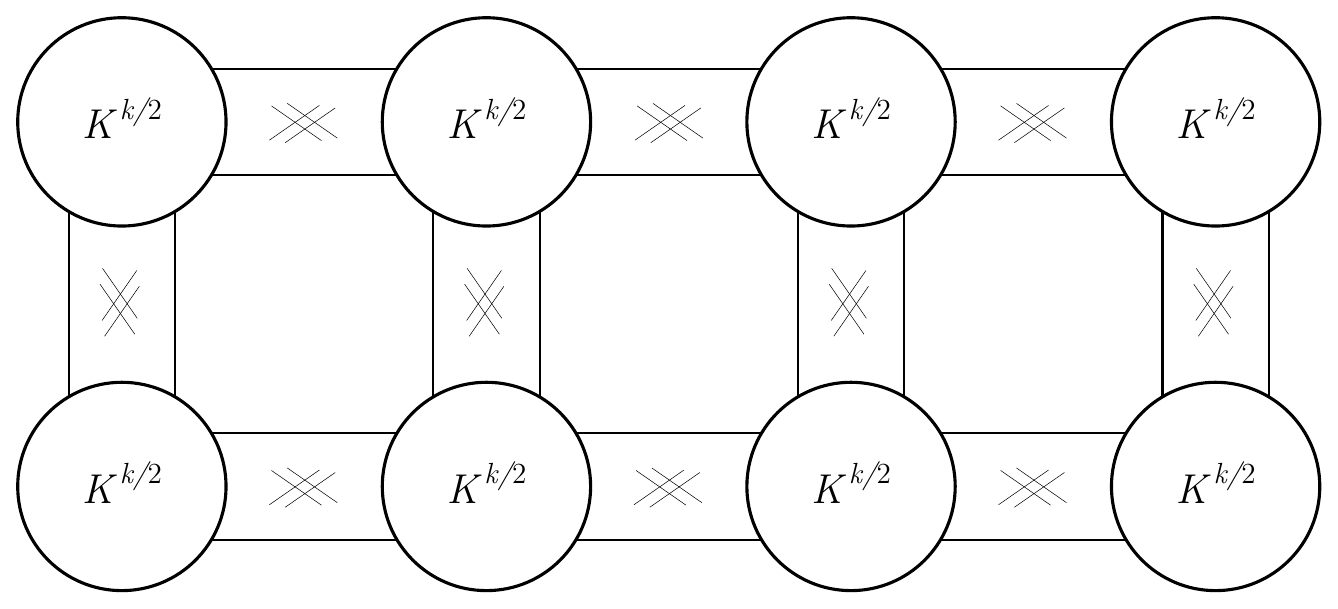}
\caption[Figure 1]{A $k$-connected graph without a $(k+1)$-block}\label{pic_Wurscht}
   \vskip-18pt\vskip-0pt
\end{center}
\end{figure}

If $k$ is odd, let $H'$ be the graph $H$ constructed above for $k-1$, and let $H$ be obtained from~$H'$ by adding a new vertex and joining it to every vertex of~$H'$. Clearly, $H$~is again $k$-connected and has minimum degree~$\lfloor \dhk-1\rfloor$, and it has no $(k+1)$-block since $H$ has no $k$-block.
\end{ex}

Our next example shows that the connectivity bound in Theorem~\ref{thm_m_wat} is sharp: we construct for every odd $k$ a $(k-1)$-connected graph $H$ of minimum degree~$\lfloor\dhk\rfloor$ whose largest $(k+1)$-blocks have size~$k+1 < \delta(H)+1$.

\begin{ex} \label{ex:m_wat}
Let $H_n$ be as in Example~\ref{lange_Wurst}. Let $H$ be obtained from~$H_n$ by replacing the degree-two vertices of~$H_n$ by complete graphs of order $(k+1)/2$ and its degree-three vertices by complete graphs of order $(k-1)/2$, joining vertices of different complete graphs whenever the corresponding vertices of $H_n$ are adjacent.
The minimum degree of this graph is $\lfloor\dhk\rfloor$, but it has only two $(k+1)$-blocks: the two $K^{k+1}$s at the extremes of the ladder.%
   \COMMENT{}
\end{ex}

We do not know whether the assumption of $k$-connectedness in Theorem~\ref{thm_m_wat} is necessary if we just want to force any $(k+1)$-block, not necessarily one of size~$\ge\delta+1$.

\medskip

If, in addition to being $k$-connected, $G$ contains no triangle, the minimum degree needed to force a $(k+1)$-block comes down to $k+1$, and the $(k+1)$-blocks we find are also larger:

\begin{thm}\label{thm_m_watK3free}
If $G$ is $k$-connected, $\delta(G)\geq k+1$, and $G$ contains no triangle, then either $V$ is a $(k+1)$-block or $G$ has at least two $(k+1)$-blocks. These can be chosen to be connected in~$G$ and of size at least $2\delta(G)$.
\end{thm}

\begin{proof}
Since $2\delta(G)>\dhk$, it suffices by Lemma~\ref{lem_sepSepLargeKBlock} to show that $|A|\ge 2\delta(G)$ for every proper $k$-separation~$(A,B)$ of~$G$. Pick a vertex $v\in A\sm B$. As $d(v)\geq k+1$, it has a neighbour $w$ in~$A\sm B$. Since $v$~and $w$ have no common neighbour, we deduce that $|A|\geq d(v)+d(w)\geq 2\delta(G)$.
\end{proof}

Any $k$-connected, $k$-regular, triangle-free graph%
   \COMMENT{}
   shows that the degree bound in Theorem~\ref{thm_m_watK3free} is sharp, because of the following observation:%
   \COMMENT{}

\begin{prop}
If $G$ is $k$-connected and $k$-regular, then $G$ has no $(k+1)$-block unless $G = K^{k+1}$ (which contains a triangle).
\end{prop}

\proof
Suppose $G$ has a $(k+1)$-block~$X$. Pick a vertex $x\in X$. The $k$ neighbours of $x$ in $G$ do not separate it from any other vertex of~$X$, so all the other vertices of $X$ are adjacent to~$x$. But then $X$ consists of precisely $x$ and its $k$ neighbours, since $|X|\ge k+1$. As this is true for every $x\in X$, it follows that $G = K^{k+1}$.
\endproof

If we strengthen our regularity assumption to transitivity (i.e., assume that for every two vertices $u,v$ there is an automorphism  mapping $u$ to~$v$),\vadjust{\penalty-500} then $G$ has no $(k+1)$-blocks, regardless of its degree:

\begin{thm}\label{trans}
If $\kappa(G)=k\ge 1$ and $G$ is transitive, then $G$ has no $(k+1)$-block unless $G=K^{k+1}$.
\end{thm}

\begin{proof}
Unless~$G$ is complete (so that $G=K^{k+1}$), it has a proper $k$-separation. Hence $V$ is not a $(k+1)$-block. Let us show that $G$ has no $(k+1)$-block at all.

If $G$ has a $(k+1)$-block, it has at least two, since $V$ is not a $(k+1)$-block but every vertex lies in a $(k+1)$-block, by transitivity. Hence any \td\ that distinguishes all the $(k+1)$-blocks of $G$ has at least two parts. By Theorem~\ref{canonicaltd} there exists such a \td~\TV, which moreover has the property that every automorphism of $G$ acts on the set of its parts. As $k\ge 1$, adjacent parts overlap in at least one vertex, so $G$ has a vertex $u$ that lies in at least two parts. But $G$ also has a vertex $v$ that lies in only one part (as long as no part of the decomposition contains another, which we may clearly assume): if $t$ is a leaf of~$\cT$ and $t'$ is its neighbour in~$\cT$, then every vertex in $V_t\sm V_{t'}$ lies in no other part than~$V_t$ (see Section~\ref{basics}). Hence no automorphism of $G$ maps $u$ to~$v$, a contradiction to the transitivity of~$G$.
\end{proof}

Theorems \ref{thm_m_wat} and \ref{trans} together imply a well-known theorem of Mader~\cite{mader70}%
  \COMMENT{}
   and Watkins~\cite{Watkins70},%
   \COMMENT{}
   which says that every transitive graph of connectivity $k$ has minimum degree at most ${\dhk-1}$.

\section{\boldmath Average degree conditions forcing a $k$-block}\label{sec_AvDeg}

As before, let us consider a non-empty graph $G=(V,E)$ fixed throughout this section. We denote its average degree by~$d(G)$. As in the previous section, we shall assume that $k\ge 0$ and consider $(k+1)$-blocks, to improve readability.

As remarked in the introduction, Mader~\cite{mader72} proved that if $d(G)\ge 4k$ then $G$ has a $(k+1)$-connected subgraph. The vertex set of such a subgraph is \lekinsep, and hence extends to a $(k+1)$-block of~$G$. Our first aim will be to show that if we seek to force a $(k+1)$-block in $G$ directly, an average degree of $d(G)\ge 3k$ will be enough.

In the proof of that theorem, we may assume that $G$ is a minimal with this property, so its proper subgraphs will all have average degrees smaller than~$3k$. The following lemma enables us to utilize this fact. Given a set $S\sub V$, write $E(S,V)$ for the set of edges of~$G$ that are incident with a vertex in~$S$.

\begin{lem}\label{esti}
If $\lambda>0$ is such that $d(G)\geq 2\lambda > d(H)$ for every proper subgraph $H\ne\es$ of~$G$, then $|E(S,V)|>\lambda|S|$ for every set $\es\ne S\subsetneq V$.
\end{lem}

\begin{proof}
Suppose there is a set $\es\ne S\subsetneq V$ such that $|E(S,V)|\leq \lambda|S|$. Then our assumptions imply
\[|E(G-S)|=|E|-|E(S,V)|\geq\lambda|V|-\lambda|S|= \lambda |V\sm S|,\]
so the proper subgraph $G-S$ of~$G$ contradicts our assumptions.
\end{proof}

\begin{thm}\label{aver_deg}
If $d(G)\ge 3k$, then $G$ has a $(k+1)$-block. This can be chosen to be connected in~$G$ and of size at least~$\delta(G)+1-k$.
\end{thm}

\begin{proof}
If $k=0$, then the assertion follows directly.
So we assume $k>0$.
Replacing $G$ with a subgraph if necessary, we may assume that $d(G)\geq 3k$ but $d(H)<3k$ for every proper subgraph $H$ of~$G$. By Lemma~\ref{esti}, this implies that $|E(S,V)|>\dhk|S|$ whenever $\es\ne S\subsetneq V$; in particular, $\delta(G)>{3\over2}k$.

Let $(\cT,\Vcal)$ be a $k$-lean \td\ of $G$, with $\cV = (V_t)_{t\in \cT}$ say. Pick a leaf $t$ of~$\cT$. (If $\cT$ has only one node, let $t$ be this node.) Then $\mathring{V}_t\ne\es$ by~(T3), since $V_t$ is not contained in any other part of~\cV.

If $|\mathring{V}_t|\leq k$ then, as also $|V_t\sm\mathring{V}_t|\leq k$,
\[\textstyle
 |E(\mathring{V}_t,V)|\leq {1\over2}|\mathring{V}_t|^2+ k\,|\mathring{V}_t|
\leq |\mathring{V}_t|\big(|\mathring{V}_t|/2+k\big)\leq \dhk\, |\mathring{V}_t|,
\]
which contradicts Lemma~\ref{esti}.%
   \COMMENT{}
   So $|\mathring{V}_t|\geq k+1\ge 2$. The set $\mathring{V}_t$ extends to a ${(k+1)}$-block~$B\sub V_t$ with the desired properties as in the proof of Theorem~\ref{min_deg}.
\end{proof}

Since our graph of Example~\ref{lange_Wurst} contains no $(k+1)$-block, its average degree is a strict lower bound for the minimum average degree that forces a $(k+1)$-block. By choosing the ladder in the construction of that graph long enough, we can make its average degree exceed $2k-1-\epsilon$ for any $\epsilon>0$. The minimum average degree that will force a $(k+1)$-block thus lies somewhere between $2k-1$ and~$3k$.

\begin{prob}\label{oque2}
Given~$k\in\N$, determine the smallest value $d_k$ of $d$ such that every graph of average degree at least~$d$ has a $k$-block.

\end{prob}

As we have seen, an average degree of $3k$ is sufficient to force a graph to contain a $(k+1)$-block.
If we ask only that the graph should have a minor that contains a $(k+1)$-block, then a smaller average degree suffices:

\begin{thm}\label{avminors}
Let $G$ be a graph with $n\geq k+1$ vertices and at least
$$(k-1)(n-(k+1))+ \binom{k+1}{2}$$
 edges. Then $G$ has a minor which has a $(k+1)$-block of size at least $\delta(G)+1$.
\end{thm}

\begin{proof}
Replacing $G$ with a minor of itself if necessary, we may assume that $G$ has at least
$(k-1) (|G|-(k+1))+ \binom{k+1}{2}$ edges but every proper minor $H$ of $G$ has less than 
$(k-1) (|H|-(k+1))+ \binom{k+1}{2}$ edges or else $G=K_{k+1}$.

We show that any two adjacent vertices $v$ and $w$ have at least $k-1$ common neighbours. If $G=K_{k+1}$, this is clear. Otherwise contracting the edge $vw$ we lose one vertex and at most $k-1$ edges. This contradicts the minimality of $G$. So $v$ and $w$ have at least $k-1$ common neighbours.

Let $({\cal T},{\cal V})$ be a $k$-lean \td\ of~$G$, with $\cV = (V_t)_{t\in {\cal T}}$ say, and let $t$ be a leaf of~${\cal T}$. (If ${\cal T}$ has only one node, let $t$ be this node.) We shall prove that $V_t$ is \lekinsep, and hence a $(k+1)$-block, in~$G$.

{\lineskiplimit-3pt
As $({\cal T},{\cal V})$ is $k$-lean, every vertex $a\in A_t := V_t\cap\bigcup_{t'\ne t} V_{t'}$ has a neighbour~$v$ in~$\mathring{V}_t$, as otherwise $X:=A_t\sm\{a\}$ would separate $A_t$ from every set $X\cup\{v\}$ with $v\in\mathring{V}_t$, which contradicts $k$-leanness since $|X\cup\{v\}|=|A_t|\le k$. As $a$ and~$v$ have $k-1$ common neighbours in~$G$, which must lie in~$V_t$, we find that every vertex in~$A_t$ has at least $k$ neighbours in~$V_t$.
Since every other vertex of~$V_t$ is incident with at least one edge, it also must has at least $k$ neighbours, which must be in~$V_t$.

}As $\mathring{V}_t\ne\es$ and hence $|V_t|\ge \delta(G)+1\ge k+1$, it suffices to show that two vertices $u,v\in V_t$ can never be separated in~$G$ by $\le k$ other vertices. But this follows from $k$-leanness: pick a set $N_u$ of $k$ neighbours of $u$ in~$V_t$ and a set $N_v$ of $k$ neighbours of~$v$ in~$V_t$ to obtain two $(k+1)$-sets $N_u\cup\{u\}$ and $N_v\cup\{v\}$ that are joined in $G$ by $k+1$ disjoint paths; hence $u$ and $v$ cannot be separated by $\le k$ vertices.
\end{proof}

Theorem~\ref{avminors} is best possible, since there are graphs with one edge
less than stated none of whose minors has a $(k+1)$-block:

\begin{ex} \label{minor_ex} 
Let $G$ be obtained by joining $n-(k-1)$ new vertices completely to a~$K_{k-1}$. This graph has $n\geq k+1$ vertices and $(k-1) (n-(k+1))+ \binom{k+1}{2}-1$ edges; this can be seen by rewriting it as a $K_{k+1}$
minus an edge $vw$ to which $n-(k+1)$ new vertices have been added, each joined to all the vertices of the $K_{k+1}$ other than $v$ and~$w$.

Let $H$ be a minor of $G$, and let $S$ be the set of those vertices of $H$ whose branch set meets the original $K_{k-1}$ in the definition of~$G$. Since all the components of $H-S$ are singleton vertices, $H$ cannot contain a $(k+1)$-block. 
  \end{ex}

   \COMMENT{}

\section{Blocks and tangles}\label{sec_Tangles}

In this section we compare $k$-blocks with tangles, as introduced by Robertson and Seymour~\cite{GMX}. Our reason for doing so is that both notions have been advanced as possible approximations to the elusive ``$(k+1)$-connected pieces'' into which one might wish to decompose a $k$-connected graph, in analogy to its tree-like block-cutvertex decomposition (for $k=1$), or to Tutte's \td\ of 2-connected graphs into 3-connected torsos (for $k=2$) \cite{ReedConnectivityMeasure, confing}.

Let us say that a set $\theta$ of separations of order at most $k$ of a graph $G=(V,E)$ is a \emph{tangle of order $k$} of~$G$ if
\begin{enumerate}[($\theta$1)]
\item\label{cT_complete} for every separation $(A,B)$ of order $< k$ of~$G$ either $(A,B)$ or $(B,A)$ is in~$\theta$;\looseness=-1
\item\label{cT_small} for all $(A_1,B_1),(A_2,B_2),(A_3,B_3) \in \theta$ we have $G[A_1]\cup G[A_2] \cup G[A_3] \neq G.$
\end{enumerate}
It is straightforward to verify that this notion of a tangle is consistent with the one given in \cite{GMX}.%
   \COMMENT{}

\goodbreak

Given a tangle~$\theta$, we think of the side $A$ of a separation $(A,B)\in\theta$ as the {\em small side\/} of~$(A,B)$, and of $B$ as its {\em large side\/}. (Thus, axiom ($\theta2$) says that $G$ is not the union of the subgraphs induced by at most three small sides.) If a set $X$ of vertices lies in the large side of every separation in~$\theta$ but not in the small side, we say that $X$ {\em gives rise to\/} or {\em defines\/} the tangle~$\theta$.

If $X$ is a \kinsep\ set of vertices, it clearly lies in exactly one of the two sides of any separation of order~${< k}$. Hence if we define $\theta$ as the set of those separations \AB\ of order~${< k}$ for which $X\sub B$, then $\theta$ satisfies~($\theta1)$, and $V$ is not a union of at most two small sides of separations in~$\theta$. But it might be the union of three small sides, and indeed $\theta$ may fail to satisfy~$(\theta2)$.%
   \COMMENT{}
   So $X$ might, or might not, define a tangle of order at most~$k$.

An $(n\times n)$-grid minor of~$G$, with $n\ge k$, also gives rise to a tangle of order~$k$ in~$G$, but in a weaker sense: for every separation \AB\ of $G$ of order less than~$k$, exactly one side meets a branch set of every {\em cross\/} of the grid, a union of one column and one row. (Indeed, since crosses are connected and every two crosses meet, we cannot have one cross in $A\sm B$ and another in~$B\sm A$.)\looseness=-1

Since $G$ can contain a large grid without containing a $k$-block (Example~\ref{grid}), it can thus have a large-order tangle but fail to have a $k$-block for any $k\ge 5$. Conversely, Examples \ref{TKn} and~\ref{ex_treeBlock} show that $G$ can have $k$-blocks for arbitrarily large~$k$ without containing any tangle (other than those of order~$\le\kappa(G)$, in which the large side of every separation is all of~$V$). For example, if $G$ is a subdivided~$K^n$ with $n\ge k+1$,%
   \COMMENT{}
   then its branch vertices form a $k$-block~$X$, but when $n\le {3\over2}(k-1)$ the separations of order~$< k$ whose large sides contain $X$ do not form a tangle, since $G$ is the union of three small sides of such separations (each with a separator consisting of two thirds of the branch vertices; compare~\cite[(4.4)]{GMX}).

Any $k$-block of size $>{3\over2}(k-1)$, however, does give rise to a tangle of order~$k$:

\begin{thm}\label{block_vs_tangle}
Every \kinsep\ set of more than $\frac{3}{2}(k-1)$ vertices in $G=(V,E)$ defines a tangle of order $k$.
\end{thm}

\begin{proof}
Let $X$ be a \kinsep\ set of more than $\frac{3}{2}(k-1)$ vertices, 
and consider the set ${\theta}$ of all separations $(A,B)$ of order less than~$k$ with $X\sub B$.
We show that ${\theta}$ is a tangle. As no two vertices of $X$ can be separated by a separation in~${\theta}$, it satisfies ($\theta$\ref{cT_complete}).
For a proof of~($\theta2$), it suffices to consider three arbitrary separations $(A_1,B_1), (A_2,B_2), (A_3,B_3)$ in~${\theta}$ and show that
\[
E(A_1)\cup E(A_2)\cup E(A_3)\not\supseteq E,\eqno(*)
\]
where $E(A_i)$ denotes the set of edges that $A_i$ spans in~$G$.

As $|X|>\frac{3}{2}(k-1)$, there is a vertex $v\in X$ that lies in at most one of the three sets ${A_i\cap B_i}$, say neither in~$A_2\cap B_2$ nor in~$A_3\cap B_3$. Let us choose $v$ in~$A_1$ if possible.%
   \COMMENT{}
  Then, as $X\sub B_1$, there is another vertex $w\ne v$ in $X\sm A_1$.
As $v$ and $w$ lie in~$X$, the set $(A_1\cap B_1)\sm\{v\}$ does not separate them.
Hence there is an edge $vu$ with $u\in B_1\sm A_1$. Since $v\notin A_2\cup A_3$,%
   \COMMENT{}
   the edge $vu$ is neither in $E(A_2)$ nor in $E(A_3)$.
But $vu$ is not in $E(A_1)$ either, as $u\in B_1\sm A_1$, completing the proof of~$(*)$.
\end{proof}

\section{\boldmath Finding $k$-blocks in polynomial time}\label{sec_Algo}

We consider graphs $G=(V,E)$, with $n$ vertices and $m$ edges, say, and positive integers $k < n$.%
   \COMMENT{}
   We shall present a simple algorithm that finds all the $k$-blocks of~$G$ in time polynomial in $n$, $m$ and~$k$.
We start our algorithm with the following step, 
which we call \emph{pre-processing}.

For two vertices $x,y$ of $G$ let $\kappa(x,y)$ denote the smallest size of a set of other vertices that separates $x$ from~$y$ in~$G$.
We construct a graph $H_k$ from~$G$ by adding, for every pair of non-adjacent vertices $x,y$, the edge $xy$ if $\kappa(x,y) \ge k$, that is, if $x$ and $y$ cannot be separated by fewer than $k$ other vertices.
Moreover, we label every non-edge $xy$ of $H_k$ by some separation of order $\kappa(x,y)< k$ that separates $x$ from~$y$ in~$G$. This completes the pre-processing.

Note that all separations of order~$< k$ of $G$ are still separations of~$H_k$,%
   \COMMENT{}
   and that the $k$-blocks of~$G$ are the vertex sets of the maximal cliques of order~$\ge k$ in~$H_k$.

\goodbreak

\begin{lem}\label{prepro_run}
 The pre-processing has running time $O(\min\{k, \sqrt n\}\cdot m\cdot n^2)$.%
   \COMMENT{}
\end{lem}

\begin{proof}
We turn the problem of finding a minimal vertex separator between two vertices into one of finding a minimal edge cut between them. This is done in the usual way (see e.g.\ Even~\cite{Even79}) by constructing a unit-capacity network $G'$ from~$G$ with $n' = 2\tilde n$ vertices and $m' = 2m + \tilde n$ directed edges, where $\tilde n = O(m)$ is the number of non-isolated vertices of~$G$.

For every non-edge $xy$ of $G$ we start Dinitz's algorithm (DA) on~$G'$, which is designed to find an $x$--$y$ separation of order $\kappa(x,y)$. If DA completes $k$ iterations of its `inner loop' (finding an augmenting path), then $\kappa(x,y) \ge k$; we then stop DA and let $xy$ be an edge of~$H_k$. Otherwise DA returns a separation \AB\ of order~$<k$; we then keep $xy$ as a non-edge of~$H_k$ and label it by \AB.
Since the inner loop has time complexity $O(m') = O(m)$ and DA has an overall time complexity of $O(\sqrt{n'} \cdot m') = O(\sqrt n \cdot m)$ (see e.g.~\cite{graphcon}), this establishes the desired bound.
\end{proof}

Now we describe the \emph{main part of the algorithm}. We shall construct a rooted tree~\cT, inductively by adding children to leaves of the tree constructed so far. We maintain two lists: a list~\cL\ of some of the leaves of the current tree, and a list~$\cB$ of subsets of~$V$. We shall change~\cL\ by either deleting its last element or replacing it with two new elements that will be its children in our tree. Whenever we add an element~$t$ to~\cL\ in this way, we assign it a set $X_t\sub V$. Think of the current list~\cL\ as containing those $t$ whose $X_t$ we still plan to scan for $k$-blocks of~$G$, and of \cB\ as the set of $k$-blocks found so far.

We start with a singleton list $\cL = (r)$ and $\cB=\es$, putting $X_r = V$.

At a given step, stop with output $\cB$ if $\cL$ is empty; otherwise consider the last element $t$ of~\cL.
If $|X_t|< k$, delete $t$ from~\cL\ and do nothing further at this step.

Assume now that $|X_t|\ge k$.
If $X_t$ induces a complete subgraph in~$H_k$, add $X_t$ to~\cB, delete $t$ from~\cL, and do nothing further at this step.

If not, find vertices $x,y\in X_t$ that are not adjacent in~$H_k$. At pre-processing, we labeled the non-edge $xy$ with
a separation $(A,B)$ of order $< k$ that separates $x$ from~$y$ in~$G$ (and in~$H_k$).
Replace $t$ in~\cL\ by two new elements $t'$ and~$t''$, making them children of~$t$ in the tree under construction, and let $X_{t'} = X_t\cap A$ and $X_{t''} = X_t\cap B$. If $|X_t| > k$, do nothing further at this step. If $|X_t|=k$, then both $X_{t'}$ and $X_{t''}$ have size~$<k$; we delete $t'$ and $t''$ again from~$\cL$ and do nothing further in this step.%
   \COMMENT{}

This completes the description of the main part of the algorithm. Let $\cT$ be the tree with root~$r$ that the algorithm constructed: its nodes are those $t$ that were in~\cL\ at some point, and its edges were defined as nodes were added to~\cL.

\begin{prop}\label{Bs_kblocks}
The main part of the algorithm stops with output $\cB$ the set of $k$-blocks of~$G$.
\end{prop}

\begin{proof}
The algorithm  stops with $\cB$ the set of vertex sets of the maximal cliques of~$H_k$ that have order~$\ge k$. These are the $k$-blocks of~$G$, by definition of~$H_k$.
\end{proof}

To analyse running time, we shall need a lemma that is easily proved by induction. A~\emph{leaf} in a rooted tree is a node that has no children, and a \emph{branching node} is one that has at least two children.%
   \COMMENT{}

\begin{lem}\label{tree_timple}
 Every rooted tree has more leaves than branching nodes.\qed
\end{lem}

\goodbreak

\begin{lem}\label{main_run}
The main part of the algorithm stops after at most $4(n-k)$ steps.
Its total running time is $O(\min\{m,n\}\cdot n^2)$.
\end{lem}

\begin{proof}
Each step takes $O(n^2)$ time, the main task being to check whether $H_k[X_t]$ is complete.%
   \COMMENT{}
   It thus suffices to show that there are no more than $4(n-k)$ steps as long as $n\le 2m$, which can be achieved by deleting isolated vertices.%
   \COMMENT{}

At every step except the last (when $\cL=\es$) we considered the last element~$t$ of~\cL, which was subsequently deleted or replaced and thus never considered again. Every such $t$ is a node of the tree~$\cT'$ obtained from~\cT\ by deleting the children of nodes~$t$ with $|X_t|=k$. (Recall that such children $t',t''$ were deleted again immediately after they were created, so they do not give rise to a step of the algorithm.) Our aim, therefore, is to show that $|\cT'|\le 4(n-k)-1$.

By Lemma~\ref{tree_timple} it suffices to show that $\cT'$ has at most $2(n-k)$ leaves. As $n\ge k+1$, this is the case if $\cT'$ consists only of its root~$r$. If not, then $r$ is a branching node of~$\cT'$. It thus suffices to show that below every branching node $t$ of $\cT'$ there are at most $2(|X_t|-k)$ leaves; for $t=r$ this will yield the desired result.%
   \COMMENT{}

By definition of~$\cT'$, branching nodes $t$ of $\cT'$ satisfy $|X_t|\ge k+1$. So our assertion holds if the two children of $t$ are leaves. Assuming inductively that the children $t'$ and~$t''$ of~$t$ satisfy the assertion (unless they are leaves), we find that, with $X_{t'} = X_t\cap A$ and $X_{t''} = X_t\cap B$ for some ${(<k)}$-separation \AB\ of~$G$ as in the description of the algorithm, the number of leaves below $t$ is at most%
   \COMMENT{}
\[
 2(|X_t\cap A|-k)+2(|X_t\cap B|-k)\leq 2(|X_t|+(k-1)-2k)\leq 2(|X_t|-k)
\]
if neither $t'$ nor~$t''$ is a leaf, and at most
 $$1+2(|X_t\cap B|-k)\le 2(|X_t|-k)$$
if $t'$ is a leaf but $t''$ is not (say), since $X_t\sm B\ne\es$ by the choice of~\AB.
\end{proof}

Putting Lemmas \ref{prepro_run} and~\ref{main_run} together, we obtain the following:

\begin{thm}\label{mainalgo}
There is an $O(\min\{k, \sqrt n\}\cdot m\cdot n^2)$-time algorithm that finds, for any graph $G$ with $n$ vertices and $m$ edges and any fixed $k < n$, all the $k$-blocks in~$G$.
\qed
\end{thm}

Our algorithm can easily be adapted to find the $k$-blocks of~$G$ for all values of~$k$ at once. To do this, we run our pre-processing just once to construct the graph~$H_n$, all whose non-edges $xy$ are labeled by an $x$--$y$ separation of minimum order and its value~$\kappa(x,y)$. We can then use this information at the start of the proof of Lemma~\ref{main_run}, when we check whether $H_k[X_t]$ is complete, leaving the running time of the main part of the algorithm at $O(n^3)$ as in Lemma~\ref{main_run}. Running it separately once for each~$k < n$, we obtain with Lemma~\ref{prepro_run}: 

\begin{thm}\label{allk}
There is an $O(\max\{m\sqrt n\, n^2,\,n^4\})$ algorithm that finds, for any graph~$G$ with $n$ vertices and $m$ edges, all the $k$-blocks of~$G$ (for all~$k$). \qed
\end{thm}

\noindent
Perhaps this running time can be improved if the trees $\cT_k$ exhibiting the $k$-blocks are constructed simultaneously, e.g.\ by using separations of order $\ell$ for all $\cT_k$ with $\ell < k$.

\medskip

The mere decision problem of whether $G$ has a $k$-block does not need our pre-processing, which makes the algorithm faster:

\begin{thm}\label{decisionproblem}
For fixed~$k$, deciding whether a graph with $n$~vertices and $m$~edges has a $k$-block has time complexity $O(mn+n^2)$.
\end{thm}

\proof
Given $k$ and a graph~$G$, we shall find either a \kinsep\ set of vertices in~$G$ (which we know extends to a $k$-block)\vadjust{\penalty-500} or a set $\cS$ of at most $2(n-k)-1$ separations of order~$<k$ such that among any $k$ vertices in $G$ some two are separated by a separation in~$\cS$ (in which case $G$ has no $k$-block).

Starting with $X=V(G)$, we pick a $k$-set of vertices in $X$ and test whether any two vertices in this set are separated by a $(<k)$-separation $(A,B)$ in~$G$. If not, we have found a \kinsep\ set of vertices and stop with a yes-answer. Otherwise we iterate with $X=A$ and $X=B$.

Every separation found by the algorithm corresponds to a branching node of~\cT. All these are nodes of~$\cT'$, of which there are at most ${4(n-k)-1}$ (see the proof of Lemma~\ref{main_run}). Testing whether a given pair of vertices is separated by some $(<k)$-separation of $G$ takes at most $k$ runs of the inner loop of Dinitz's algorithm (which takes $O(m+n)$ time), and we test at most ${k\choose 2}$ pairs of vertices in~$X$.%
   \COMMENT{}
\endproof

Let us say that a set $\cS$ of ($<k)$-separations in $G$ {\it witnesses\/} that $G$ has no $k$-block if among every $k$ vertices of $G$ some two are separated by a separation in~$\cS$. Trivially, if $G$ has no $k$-block then this is witnessed by some $O(n^2)$ separations. The proof of Theorem~\ref{decisionproblem} shows that this bound can be made linear:

\begin{cor}\label{certificate}
Whenever a graph of order~$n$ has no $k$-block, there is a set of at most $4(n-k)-1$ separations witnessing this.\qed
\end{cor}

Let us call any tree $\cT$ as in our main algorithm (at any stage), with each of its branching nodes~$t$ labelled by a separation $(A,B)_t$ of $G$ that separates some two vertices of~$X_t$, a \emph{block-decomposition} of~$G$. The sets $X_t$ with $t$ a leaf will be called its {\em leaf sets\/}.

The {\em adhesion\/} of a block-decomposition is the maximum order of the separations~$(A,B)_t$. A block-decomposition is {\em $k$-complete\/} if it has adhesion~$<k$ and every leaf set is \kinsep\ or has size~$<k$. The {\em width\/} of a block-decomposition is the maximum order of a leaf set~$X_t$. The {\em block-width\/} ${\rm bw}(G)$ of~$G$ is the least $k$ such that $G$ has a block-decomposition of adhesion and width both at most~$k$.

Having block-width $<k$ can be viewed as dual to containing a $k$-block, much as having tree-width~$<k-1$ is dual to containing a haven or bramble of order~$k$, and having branch-width~$<k$ is dual to containing a tangle of order~$k$. Indeed, we have shown the following:

\begin{thm}\label{dualitythm}
Let $\cD = (\cT; (A,B)_t\,,\, t\in\cT)$ be a block-decomposition of a graph~$G$, and let $k\in\N$.
\begin{enumerate}[\rm(i)]\itemsep=0pt\vskip-\smallskipamount\vskip0pt
\item Every edge of $G$ has both ends in some leaf set of~$\cT$.
\item If $\cD$ has adhesion~$<k$, then any $k$-block of~$G$ is contained in a leaf set of~$\cT$.
\item If \cD\ is $k$-complete, then every $k$-block of~$G$ is a leaf set, and all other leaf sets have size~$<k$.\qed
\end{enumerate}
\end{thm}

Theorem~\ref{dualitythm} implies that $G$ has a block-decomposition of adhesion and width both at most~$k$ if and only if $G$ has no $(k+1)$-block. The least such~$k$ clearly equals the greatest~$k$ such that $G$ has a $k$-block, its {\em block number\/} $\beta(G)$:

\begin{cor}
Every finite graph $G$ satisfies $\beta(G) = {\rm bw}(G)$.\qed
\end{cor}

By Theorem~\ref{decisionproblem} and its proof, we obtain the following complexity bound:

\begin{cor}\label{blockwidthcomplexity}
Deciding whether a graph with $n$ vertices and $m$ edges has block-width~$<k$, for $k$ fixed, has time complexity $O(mn+n^2)$.\qed
\end{cor}

For $k$~variable, the proof of Theorem~\ref{decisionproblem} yields a complexity of $O(k^3(m+n)\allowbreak(n-k))$. Alternatively, we can use pre-processing to obtain $O(\min\{k, \sqrt n\}\cdot m\cdot n^2)$ by Theorem~\ref{mainalgo}.

\medbreak

The above duality between the block number and the block-width of a graph is formally reminiscent of the various known dualites for other width parameters, such as the tree-width, branch-width, path-width, rank-width, carving-width or clique-width of a graph. The `width' to which these parameters refer, however, is usually that of a tree-like decomposition of the graph itself, which exhibits that it structurally resembles that tree. In our block-decompositions, on the other hand, the tree $T$ merely indicates a recursion by which the graph can be decomposed into small sets: the separations used to achieve this, though of small order, will not in general be nested, and the structure of $G$ will not in any intuitive sense be similar to that of~$T$.

In~\cite{DiestelOumDualityII}, Diestel and Oum give a structural duality theorem for $k$-blocks in the sense of those traditional width parameters. The graph structure that is shown to witness the absence of a $k$-block is not a tree-structure, but one modelled on more general (though still tree-like) graphs. Whether or not a structural duality between $k$-blocks and tree-like decompositions exists remains an open problem. It has been formalized, and stated explicitly~\cite{DiestelOumDualityII}, with reference to a fundamental structural duality theorem between tangle-like `dense objects' and tree-like decompositions, which implies all the traditional duality theorems for width parameters~\cite{DiestelOumDualityI} but does not yield a duality theorem for $k$-blocks.


\section{Further examples}\label{sec_FurtherEx}

In this section we discuss several examples dealing with certain situations of our results.
In particular, we will describe one example that shows that the $(k+1)$-block found in Theorem~\ref{min_deg} need not contain any vertex of the adhesion set that lies in the same part of the \td,\ and we will describe two examples dealing with the notion of \Tshaped-shaped and Lemma~\ref{lem_T-shaped}.
All these examples are included only in this extended version of this paper.

Recall that in the proof of Theorem~\ref{min_deg} we considered a $k$-lean tree 
\td\ \TV\ of a graph $G$ with $\delta(G)\geq 2k$ and showed for each leaf~$t$ of~\cT\ that $V_t$
includes a $(k+1)$-block $b$.
We now give an example where the adhesion set $V_t\cap V_{t'}$ lies completely outside~$b$, where $t'$ is the neighbour of~$t$ in~\cT.

\begin{figure}[h]
\begin{center}
\includegraphics[width=.6\textwidth]{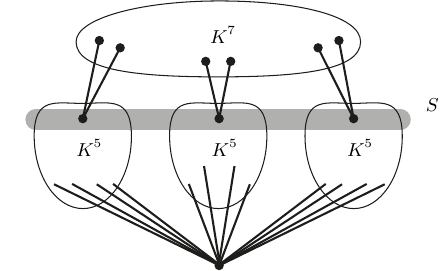}
\caption[Figure 1]{$S$ lies outside the $4$-block containing the $K^7$}\label{pic_SepOutOfBlock}
   \vskip-17pt\vskip-12pt
\end{center}
\end{figure}

\begin{ex}
Let $G$ be the graph in Figure~\ref{pic_SepOutOfBlock} and let \TV\ be the \td\ with adhesion sets $S$ and those $2$-separators that contain one vertex in~$S$ and the lowest vertex.
So \cT\ is a star with $4$ leaves.
It is not hard to show that \TV\ is $3$-lean.
For every vertex $x$ of the adhesion set $S$ inside the upper part $V_t$, its two neighbours in~$V_t$ together with the bottom vertex separate it from any vertex in~$\mathring{V}_t$ but its neighbours. Hence $x$ does not lie in the $4$-block $b$ that contains~$\mathring{V}_t$.
As no vertex of~$S$ lies in~$b$, we conclude $\mathring{V}_t=b$.
\end{ex}

Our next example shows that a local version of Lemma~\ref{lem_T-shaped} as discussed just before the lemma is false. We considered there the question of whether every proper $k$-separation \AB\ in a $k$-connected graph such that $A$ contains no $(k+1)$-block must be \Tshaped-shaped, at least if $A$ is minimal as above.

\begin{ex}\label{Pi}
Let $k=6$, and let $G$ be the complement of the disjoint union of three induced paths $P_1,P_2,P_3$ of length~$2$. Then each of three sets $V(P_i)$ is separated by the union of the other two. Hence any 7-block misses a vertex from each~$P_i$ and thus has at most 6 vertices. Hence, $G$~has no 7-block.

But $G$ is $6$-connected, and its only proper $6$-separations \AB\ have the form that either $A\sm B$ consists of the ends of some $P_i$ and $B\sm A$ of its inner vertex, or vice versa.
Let \AB\ be a $6$-separation of the first kind.
Obviously, $A$ is minimal such that \AB, for some~$B$, is a proper $6$-separation.

To show that $(A,B)$ is not \Tshaped-shaped, suppose it is, and let this be witnessed by another proper $6$-separation~$(C,D)$. Then $(C,D)$ is neither $(A,B)$ nor~$(B,A)$.%
   \COMMENT{}
   So the separators $A\cap B$ and $C\cap D$ meet in exactly one~$V(P_i)$, say in~$V(P_1)$. Then ${C\cap D}$ contains~$V(P_2)$, say, while $A\cap B$ contains~$V(P_3)$. By assumption, the ends of $P_2$ lie in~$A\sm B$. If the ends of $P_3$ lie in~$C\sm D$, say, we have $|A\cap C| = 7$. This contradicts the choice of~$(C,D)$, so $(A,B)$ is not \Tshaped-shaped.
\end{ex}

So our envisaged local version of Lemma~\ref{lem_T-shaped} is false. Since $|A| = 8\leq\dhk$ in the above example, we could not simply use Lemma~\ref{link_small} to show that $(A,B)$ is not \Tshaped-shaped. In our next example $A$~is larger, so that we can.

\begin{ex}\label{bigA}
Let $G$ be the graph of Figure~\ref{pic_ExTshaped}.
It is 5-connected but has no $6$-block.
Let $A$ be the vertex set that consists of the vertices of the upper three~$K^5$s, and let $B$ be the union of the vertex sets of the lower three complete graphs.
Then $(A,B)$ is a proper $5$-separation, with $A$ minimal.
By Lemma~\ref{link_small}, \AB\ is not \Tshaped-shaped.%
   \COMMENT{}
\end{ex}

By Lemma~\ref{lem_T-shaped}, however, both these examples must have some \Tshaped-shaped $k$-separation. In Example~\ref{Pi}, the separation~$(B,A)$ is \Tshaped-shaped. In Example~\ref{bigA}, the separation $(A',B')$ where $A'$ consists of the two leftmost complete graphs and $B'$ of the other four, is \Tshaped-shaped.

\begin{figure}[h]
\begin{center}
\includegraphics[width=.45\textwidth]{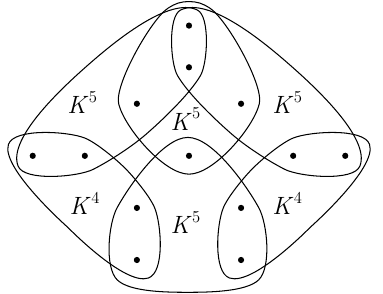}
\caption[Figure 1]{A 5-connected graph without a $6$-block}\label{pic_ExTshaped}
   \vskip-15pt\vskip-12pt
\end{center}
\end{figure}

\section{Acknowledgements}

We thank%
   \COMMENT{}
   Jens Schmidt for pointing out reference~\cite{mader74} for Theorem~\ref{min_deg}, Matthias Kriesell for advice on the connectivity of triangle-free graphs, Paul Seymour for suggesting an algorithm of the kind indicated in Section~\ref{sec_Algo}, and Sang-il Oum for pointing out Corollary~\ref{certificate} and Theorem~\ref{dualitythm}.

\bibliographystyle{plain}
\bibliography{collective}

\end{document}